\documentclass[11pt]{amsart}
\usepackage{hyperref}
\usepackage{tikz}
\usepackage[active]{srcltx}
\usepackage{xypic}
\usepackage{graphicx}
\usepackage{tikz-cd} 
\usepackage{amsmath}

\newtheorem{theorem}{Theorem}[section]
\newtheorem{lemma}[theorem]{Lemma}
\newtheorem{proposition}[theorem]{Proposition}
\newtheorem{corollary}[theorem]{Corollary}

\newtheorem{definition}[theorem]{Definition}
				
\theoremstyle{remark}
\newtheorem{example}[theorem]{\bf Example}

\theoremstyle{plain}

\newcommand{\x}{\boldsymbol{x}}
\newcommand{\g}{\boldsymbol{g}}

\newcommand{\p}{\partial}

\newcommand{\s}{\Sigma}
\newcommand{\N}{\Bbb N}

\newcommand{\pp}{\boldsymbol{P}}
\newcommand{\q}{\boldsymbol{Q}}
\newcommand{\bnd}{\boldsymbol{bnd}}
\newcommand{\la}{\lambda}

\title{Random walks of infinite moment on free semigroups}
\author[B. Forghani]{Behrang Forghani}
\address{Department of Mathematics\\
University of Connecticut\\
341 Mansfield Rd, Storrs, CT, USA}
\email{\href{mailto:behrang.forghani@uconn.edu}{behrang.forghani@uconn.edu}}

\author[G. Tiozzo]{Giulio Tiozzo}
\address{Department of Mathematics\\ 
University of Toronto\\ 
40 St George St\\ 
Toronto, ON, Canada\\}
\email{\href{mailto:tiozzo@math.toronto.edu}{tiozzo@math.toronto.edu}}

\begin{document}

\maketitle

\begin{abstract}
We consider random walks on finitely or countably generated free semigroups, and identify their Poisson boundaries for classes of measures 
which fail to meet the classical entropy criteria.  
In particular, we introduce the notion of $w$--logarithmic moment,  and we show that if a random walk on a free semigroup has \emph{either} finite entropy \emph{or} finite $w$-logarithmic moment for some word $w$, then the space of infinite words with the resulting hitting measure is the Poisson boundary.
\end{abstract}

\section{Introduction}
The notion of Poisson boundary for Markov chains goes back to
the work of Feller \cite{Feller} and Blackwell \cite{Blackwell1955}, 
who showed that the Poisson boundary of random walks on an abelian group is always trivial (i.e. a singleton). Their work did not get much attention, 
since in all known examples the boundary was trivial. Then in the 1960's, Furstenberg showed that the Poisson boundary for a random walk on a non-amenable group is non-trivial, and identified the boundary for certain random walks on lattices in Lie groups. 
He then employed the theory of Poisson boundary 
to prove several fundamental rigidity results for lattices in Lie groups (see \cite{Fu70}). 

In general, given a pair $(G, \mu)$, where $G$ is a group (or semigroup) and $\mu$ a probability measure on $G$, the main question is to identify 
the Poisson boundary, which is always defined as an abstract measure space, with a concrete boundary of the group given e.g. by a topological compactification. More precisely, in many cases one can prove that the random walk on $G$ with distribution $\mu$ converges almost surely 
in a suitable topological boundary $\partial G$, hence $\partial G$ is equipped with the \emph{hitting measure} $\lambda$ of the random walk. 
Then the question becomes whether the pair $(\partial G, \lambda)$ is the Poisson boundary of the random walk $(G, \mu)$.

This question has been studied for almost 50 years for a large number of different groups. 
One of the first examples of an explicit non-trivial boundary has been established by Dynkin-Maljutov \cite{Dynkin-Maljutov61}, who identified the Poisson boundary of  a first neighbor random walk on a free group of rank $2$ with the space of reduced infinite words. 
For hyperbolic groups, the identification goes back to Ancona \cite{Ancona88} when $\mu$ is finitely supported. 

It is important to point out that the Poisson boundary is an invariant of the \emph{pair} $(G, \mu)$ and as such it may vary greatly for different measures even on the same group. In particular, by the work of Kaimanovich-Vershik \cite{K-Vershik83} and Derriennic \cite{Der80} the triviality of the boundary has been linked to the vanishing 
of the \emph{asymptotic entropy}. This \emph{entropy criterion} has been then extended by Kaimanovich, who formulated 
geometric criteria (the \emph{ray criterion} \cite{K85} and the \emph{strip criterion} \cite{K00}) to identify the Poisson boundary. These techniques have been widely applied 
to many types of groups, such as e.g. hyperbolic and relatively hyperbolic groups, lattices in Lie groups, lamplighter groups, and more recently mapping class groups, or the group of outer automorphisms of the free group (see \cite{Erschler2010} for a survey).

However, all these results are based on the classical hypotheses for the application of the strip criterion, namely that the measure $\mu$ has 
\emph{finite entropy} and \emph{finite logarithmic moment}. In this paper, we will go beyond such restrictions for random walks on the \emph{free semigroup}. 

In fact, even though the free semigroup is arguably the simplest possible case, it is still an open conjecture that 
\emph{the Poisson boundary for \textup{any} generating measure on the free semigroup can be identified with the space of infinite words}. 

In recent work, Kaimanovich and the first author \cite{BK2013} have proved the conjecture for the free semigroup in the case when $\mu$ has finite logarithmic moment, without any assumption on the entropy. 

In this paper, we will extend these results to a much larger class of measures $\mu$ on a free semigroup $\Sigma$ of finite or countable rank. One of our main results is the following. 

\begin{theorem}\label{thm:main either or}
Let $\mu$ be a generating measure on a free semigroup $\Sigma$ of finite or countable rank, and let $(\p \s, \lambda)$ denote the space of infinite words 
in the generators, with the hitting measure for the random walk. If $\mu$ has {\bf either} finite entropy {\bf or} finite logarithmic moment, then $(\p \s,\la)$ is the Poisson boundary of the random walk $(\s,\mu)$.
\end{theorem}

As we mentioned, the strip approximation cannot be used in this context, hence we develop different tools. 
The main idea is that the distance from the identity yields a projection $\Sigma \to \mathbb{N}$, and random walks on $\mathbb{N}$ have 
trivial boundary. This can be used, by framing the problem in terms of  random walks on equivalence classes, to prove that the relative entropy of the original walk is zero almost surely. 

In order to state the second main result, let us fix a finite word $w \neq e$ in $\Sigma$. We define the $w$-norm $|g|_w$ of an element $g \in \Sigma$ 
as the number of subwords of $g$ which are equal to $w$ (see Section \ref{S:w} for the precise definition). We say that the measure $\mu$ has finite logarithmic $w$-moment if $\int_\Sigma \log |g|_w \ d \mu(g) < \infty$.
The second main result is the following.

\begin{theorem}\label{thm:finite w norm}
Suppose that there exists a word $w \neq e$ in $\Sigma$ such that $\mu$ has finite logarithmic $w$-moment. Then, the space $(\p \s, \la)$ of infinite words is the Poisson boundary of the random walk $(\s,\mu)$.
\end{theorem}

Such a criterion can be quite flexible, as we will show in the next few examples, which 
were not available to the older techniques.

\begin{example}
Let $\Sigma = \langle a, b \rangle $ be a free semigroup of rank $2$ with generators $a, b$, and let us consider the probability measure $\mu$ defined as
$$\begin{array}{ll} 
\mu(a^{2^k}) = \frac{c}{k^2} & \textup{for }k \geq 1 \\
\mu(b) = \frac{1}{2}  
\end{array}$$
where $c$ is a constant such that the total measure is $\frac{1}{2} + \sum_{k=1} \frac{c}{k^2} = 1$. This measure clearly has infinite logarithmic moment, as 
$\sum_k \frac{\log(2^k)}{k^2} = \sum_k \frac{\log 2}{k} = + \infty$. However, it has finite entropy, so we can obtain the Poisson boundary as a Corollary of Theorem \ref{thm:main either or}. 
\end{example}

However, there are random walks on a free semigroup whose both entropy and logarithmic moment are not finite, as in the following case.
\begin{example}
Let $\Sigma = \langle a, b \rangle $ be a free semigroup of rank $2$ with generators $a, b$, and let us consider the probability measure $\mu$ defined as
$$\begin{array}{ll} 
\mu(a^{2^k}) = \frac{c}{k^2} & \textup{for }k \geq 1 \\
\mu(a^{3^k})=d_k &  \textup{for }k \geq 1 \\
\mu(b) = \frac{1}{2} .
\end{array}$$
We choose $d_k>0$ such that $\sum_kd_k\log{d_k}=-\infty$ and $\sum_k (d_k+\frac{c}{k^2})=\frac{1}{2}$. Therefore, both  entropy and the logarithmic moment are infinite, hence Theorem \ref{thm:main either or} cannot be applied.
However, if we take $b = w$, then it is extremely easy to see that the logarithmic $b$-norm of $\mu$ is finite, hence we can identify the Poisson boundary by using Theorem \ref{thm:finite w norm}.
\end{example}

Note that one can use this result for many distributions with arbitrarily ``fat" tails, by choosing carefully the word $w$. For instance:

 \begin{example}
 Let $\Sigma = \langle a, b \rangle $, and consider \emph{any} sequences $\{c_k\}_{k \geq 1}, \{d_k \}_{k \geq 1}$ of positive numbers such that $\sum_k c_k + \sum_k d_k = 1$. Consider 
 the measure $\mu$ defined as 
 $$\begin{array}{ll} 
\mu(a^{k}) = c_k \\
\mu(b^{k})=d_k.
\end{array}$$
Then, the Poisson boundary of $(\Sigma, \mu)$ is the same as the Poisson boundary of $(\Sigma, \mu^2)$. Moreover, the word $w = ab$ 
lies in the support of $\mu^2$, and each element in the support of $\mu^2$ contains $ab$ at most once, hence the $w$-moment of $\mu^2$ is 
finite. Hence, the Poisson boundary of $(\Sigma, \mu)$ is the space of infinite words.
 \end{example}
 
 \subsection{Structure of the paper}
In Sections \ref{S:not}, \ref{S:mub}, and \ref{S:entro} we will recall the basic definitions about boundaries of random walks, as well as the definition of random walks on equivalence classes and the criteria for boundary triviality. The proof of Theorem \ref{thm:main either or} is contained in section \ref{S:main1}. Then in section \ref{S:stop} we will discuss stopping times, and we will use them in section \ref{S:main2} to prove Theorem \ref{thm:finite w norm}. 

\subsection{Acknowledgments} 
We would like to thank L. Bowen and V. Kaimanovich for fruitful discussions. G. T. is partially supported by NSERC and the Connaught fund.

\section{Notation and background material} \label{S:not}
Let $W$ be a finite or countable non-empty set, and $\s$ be the semigroup with identity freely generated by $W$. 
Thus, every element of $\s$ can be written uniquely as $g = w_1 w_2 \dots w_n$ where $w_i \in W$, with the degenerate case of the identity element $e$ which is represented by the empty word. 
Each element of $\s$ is called a  \emph{finite word}, and the number $n$ is called the \emph{word metric (length)} of $g$ and is denoted by $|w|$.
If $x=w_1\cdots w_n$ and $y=w_1\cdots w_m$ are finite words and $m\geq n$, then we also define $x^{-1}y:=w_{n+1}\cdots w_m$, and $e^{-1}=e$. 

Let  $\mu$ be a probability measure on $\s$. 
We will assume that $\mu$ is \emph{generating}, i.e. the semigroup with identity generated by the support of $\mu$ equals $\s$.
Note that this is not a restrictive condition, as if $\mu$ is not generating, then the semigroup $\Sigma^+$ generated by the support $\mu$ is also free, 
hence one can just replace $\Sigma$ by  $\Sigma^+$.
Let us denote by $\mu^{\star k}$ the $k^{th}$ fold convolution of $\mu$, that is for  any finite word $w$,
$$
\mu^{\star k}(w)=\sum_{g_1\cdots g_k=w}\mu(g_1)\cdots\mu(g_k).
$$
For  a finite word $w$ in $\s$ define the transition probability 
$$
p(w,wg):=\mu(g).
$$
The Markov process associated with $p$ is called \emph{random walk} $(\s,\mu)$.
Let $\s^\N$ be the set of infinite sequences of elements of $\s$, which is equipped with the product measure $\mu^\N$.  The probability space $(\s^\N,\mu^\N)$ is called the \emph{space of increments} for the random walk $(\s,\mu)$.  Let $\Omega=\s\times\s^\N$ and for any $g \in \s$ define the map
$$
\begin{array}{c}
\s^\N\to\Omega\\
\{g_n\}_{n\geq1}\to \{x_n\}_{n\geq0}
\end{array}
$$
where $x_0:=g$ and $x_n=x_0g_1\cdots g_n$ for $n\geq 1$. The $\s$--valued map $x_n$ is called the \empty{position of random walk} at time $n$. The image of the probability measure $\mu^\N$ under the preceding map is denoted by $\pp_g$. The probability space $(\Omega,\pp_g)$ is called the \emph{space of sample paths} started from $g$. 
Let us also denote as  $U : \Sigma^\mathbb{N} \to \Sigma^\mathbb{N}$ the shift on the space of increments.

\subsection{Poisson boundary}
Let $m$ be a probability measure supported on $\s$, that is $m(g)>0$ for any $g$ in $\s$. 
Let us define 
$$\pp_m=\sum_gm(g)\pp_g .$$  
We say two sample paths $\{x_n\}_{n\geq0}$ and $\{y_n\}_{n\geq0}$ are equivalent whenever they coincide after  finite time shifts; more precisely, if there are two integers $i$ and $j$ such that $x_{n+i}=y_{n+j}$ for $n\ge 0$. Consider the $\sigma$-algebra $\mathcal{A}$ of all measurable unions of these equivalence classes (mod 0) with respect to probability measure $\pp_m$. By Rokhlin's theory of Lebesgue spaces \cite{Rokhlin52}, there exist a unique (up to isomorphism) measurable space $\Gamma = \Gamma(\s, \mu)$ and a measurable map $\boldsymbol{bnd}:\s^{\Bbb N}\to \Gamma$, called the \emph{boundary map}, such that the $\sigma$-algebra $\mathcal{A}$ coincides (mod 0) with the $\sigma$-algebra of $\boldsymbol{bnd}$-preimages of measurable subsets of $\Gamma$.

\begin{definition}
Let $\pp:=\pp_e$ be the probability measure on the sample space with respect to the random walk $(\s,\mu)$ with initial distribution concentrated on the identity element $e$ of $\s$.  The probability space $(\Gamma,\nu)$ is called the \emph{Poisson boundary} of the random walk $(\s,\mu)$, where $\nu=\boldsymbol{bnd}_\star(\pp)$ is the image of the probability measure $\pp$ under the measurable map $\boldsymbol{bnd}$, which is called the \emph{harmonic measure}.
\end{definition}

\subsection{Harmonic functions}
Let $f : \Sigma \to \mathbb{R}$ be a bounded real-valued function.  For any finite word $g$ in $\s$, let us define the action of $\mu$ on $f$ as 
$$
\mu \cdot f (g) := \sum_y\mu(y)f(gy).
$$
A bounded function is called $\mu$--\emph{harmonic} if it is invariant under the action of $\mu$, that is $f=\mu\cdot f$. Let us denote the space of all bounded $\mu$--harmonic function as $H^\infty(\mu)$, which is a Banach space when is equipped with the  supremum norm. One can show that the Poisson boundary $(\Gamma, \nu)$ is related to bounded harmonic functions via the \emph{Poisson representation formula} (see e.g. \cite{Fu70}), which establishes 
an isometric isomorphism between $H^\infty(\mu)$ and $L^\infty(\Gamma, \nu)$. 
More precisely, when $\hat{f}$ is in $L^\infty(\Gamma,\nu)$ then $f(g)=\int\hat{f} \ dg\la$ is a bounded $\mu$--harmonic function.  When $f$ is in $H^\infty(\mu)$, then
$\hat{f}(\bnd(\x))=\lim_nf(x_n)$ exists for almost every sample path $\x=\{x_n\}_{n\geq0}$ and belongs to $L^\infty(\Gamma,\nu)$. These two maps are inverses to each other and preserve norms, establishing the isomorphism.

Finally, a positive harmonic function $f : \Sigma \to \mathbb{R}$ is \emph{minimal} if any positive harmonic 
function $g$ such that $f(x) \geq g(x)$ for every $x$ must be a scalar multiple of $f$.

\section{$\mu$--boundaries and conditional random walks} \label{S:mub}

A probability space $(B, \lambda)$ obtained by taking the quotient of the Poisson boundary with respect to a $\s$--invariant measurable partition is called a
\emph{$\mu$--boundary}. Let us denote the corresponding quotient map by 
$$
\Lambda:\Gamma \to B.
$$

Let $(B, \lambda)$ be a $\mu$-boundary. For each $\xi \in B$, the \emph{conditional random walk} associated with $\xi$ is defined as the Markov process on $\s$ 
with transition probabilities
\begin{equation} \label{E:cond}
p^\xi(x,xg) = \mu(g) \frac{d xg \lambda}{d x \lambda}(\xi).
\end{equation}
Denote by $\pp^\xi$ the probability measure on the space of sample paths with respect to the Markov process associated to $p^\xi$.
One should think of this process as the random walk conditioned to hitting the boundary at $\xi$.
For each $\xi$, the \emph{relative Poisson boundary} is the Poisson boundary of the Markov process $p^\xi$. By disintegration, we can write
\begin{equation}\label{eq:disintegration}
\pp=\int_{B}\pp^\xi \ d\lambda(\xi).
\end{equation}

We now recall two important lemmas which link minimal harmonic functions, Poisson boundary, and the conditional random walks. 
They are probably well-known, but we provide proofs for completeness.

\begin{lemma}
Let $(B, \lambda)$ be a $\mu$-boundary of the random walk $(\s, \mu)$. Then: 
\begin{enumerate}
\item for $\lambda$-almost every $\xi \in B$, the function 
$$u^\xi(g) := \frac{ d g \lambda}{ d \lambda}(\xi)$$
is harmonic;
\item the Poisson boundary for the conditional random walk $p^\xi$ is trivial if and only if the function $u^\xi$ is minimal harmonic.
\end{enumerate}
\end{lemma}

\begin{proof}
(1) Since $\lambda$ is a stationary measure, then 
$$ \lambda = \sum_{h \in \s} \mu(h) \ h\lambda.$$
Hence, by acting by $g$ on both sides we get  
$$ g\lambda = \sum_{h \in \s} \mu(h) \ gh\lambda$$
hence, taking the Radon-Nykodym derivative with respect to $\lambda$ we get for almost every $\xi \in B$
$$ u^\xi(g) =  \frac{d g\lambda}{d \lambda}(\xi) = \sum_{h \in \s} \mu(h) \ \frac{d (gh\lambda)}{d\lambda}(\xi) = \sum_{h \in \s} \mu(h) \ u^\xi(gh)$$
hence $u^\xi$ is harmonic.

(2) Recall that a function $f$ is $p^\xi$-harmonic if and only if for each $g$ 
$$f(g) = \sum_h \mu(g^{-1} h) \frac{dgh\lambda}{dg\lambda}(\xi) f(gh) =  \sum_h \mu(g^{-1} h) \frac{u^\xi(gh)}{u^\xi(g)} f(gh)$$
which implies 
$$f(g) u^\xi(g) = \sum_h \mu(g^{-1} h)  f(gh) u^\xi(gh)$$
Hence, $f$ is $p^\xi$-harmonic if and only if $v(g) = f(g) u^\xi(g)$ is $\mu$-harmonic. 
Thus, let us assume that the Poisson boundary of $p^\xi$ is trivial, and let $v$ be a $\mu$-harmonic function such that $v \leq u^\xi$. 
Then by the above observation the function $f(g) = \frac{u^\xi(g)}{v(g)}$ is $p^\xi$-harmonic and bounded, hence it must be constant. 
Thus, $u^\xi = c v$, so $u^\xi$ is minimal. Conversely, if $u^\xi$ is minimal, then for each function $f$ which is bounded and $p^\xi$-harmonic, 
the function $v = f u^\xi$ is $\mu$-harmonic and bounded above by a multiple of $u^\xi$, hence $v = c u^\xi$ and $f$ is constant. 
\end{proof}

\begin{lemma}\label{lem:maximality}
A $\mu$--boundary $(B,\lambda)$ is the Poisson boundary if and only if the Poisson boundaries of the conditional random walks are almost surely trivial.
\end{lemma}

\begin{proof}
Let $\xi \in B$ be a boundary point, and define the function 
$$u^\xi(g) := \frac{d g \lambda}{d \lambda}(\xi).$$
For almost every $\xi$, the function $u^\xi$ is harmonic. Moreover, since the Poisson boundary of the conditional Markov process $P^\xi$ is almost surely trivial, then for almost every $\xi$ the function $u^\xi$ is minimal.

Now, each minimal harmonic function is a Martin kernel (\cite{Woess09}, Theorem 7.50), and the Poisson boundary can be realized as a full measure subset of the Martin boundary (\cite{Woess09}, Section 7.E); hence, for almost every $\xi \in B$ there exists a point $\gamma \in (\partial \Sigma, \nu)$ 
such that 
\begin{equation}\label{eq:minimal partition}
u^\xi(g)=\frac{dg\nu}{d\nu}(\gamma).
\end{equation}
Thus, by definition the Markov processes $p^\xi$ and $p^\gamma$ coincide, hence the measures $\pp^\xi$ and $\pp^\gamma$ on the space of sample paths coincide. 
Let us now consider the quotient map $\Lambda:\p \s\to B$. By definition,  
$$\pp^\xi = \int _{\Lambda^{-1}(\xi)} \pp^\gamma \ d\nu_\xi(\gamma)$$
where $\nu_\xi$ is the induced measure on the fiber $\Lambda^{-1}(\xi)$ of the projection. Finally, let us note that by construction any two measures $\pp^\gamma$ and $\pp^{\gamma'}$ for $\gamma \neq \gamma'$ in $\partial \Sigma$ are mutually singular. Thus, since $\pp^\xi = \pp^\gamma$ we have that 
the measure $\nu_\xi$ must be atomic, hence $\Lambda^{-1}(\xi)$ is a singleton. Since this is true for almost every $\xi \in B$, the map $\Lambda$ is a $\Sigma$-equivariant measurable isomorphism, which proves the claim. 
\end{proof}

\section{Shannon entropy and relative entropy} \label{S:entro}
Let $\zeta=\{\zeta_i : i\geq 1\}$ be a countable partition of the sample space $\Omega$ of the random walk $(\s,\mu)$ into measurable sets. 
The \emph{entropy} (Shannon entropy) of $\zeta$ is defined as 
$$
H_{\pp}(\zeta):=-\sum_i\pp(\zeta_i) \log\pp(\zeta_i),
$$
where we take $0\log0:=0$.
Let $\alpha_k$ be the pointwise partition with respect to the $k^{th}$ position of the random walk $(\s,\mu)$; that is, two sample paths $\x$ and $\x'$ are $\alpha_k$--equivalent if and only if $x_k=x'_k$; therefore,
$$
H_{\pp}(\alpha_k)=-\sum_g\mu^{\star k}(g)\log{\mu^{\star k}(g)}.
$$
Note that $H_{\pp}(\alpha_k)$ sometimes is denoted by $H(\mu^{\star k})$.  
Since the sequence $\{H_{\pp}(\alpha_k)\}_{k\geq1}$ is subadditive, the limit $\frac{1}{k}H_{\pp}(\alpha_k)$ exists and is called the \emph{asymptotic entropy} of the random walk $(\s,\mu)$.
\begin{definition}
We say $\mu$ has finite entropy if $H_{\pp}(\alpha_1)$ is finite. 
\end{definition}

The following theorem is a special case of the entropy criterion due to Kaimanovich-Vershik \cite{K-Vershik83} and Derriennic \cite{Der80}: since 
the Poisson boundary for random walks on abelian groups is trivial, then the asymptotic entropy vanishes.
\begin{theorem}\label{thm:abelian zero entropy}
Let $\s=\mathbb{N}$, and suppose that $H_{\pp}(\alpha_1)$ is finite. Then 
$$
\lim_k\frac{1}{k}H_{\pp}(\alpha_k)=0.
$$
\end{theorem}
For a given $\mu$--boundary $(B,\la)$, two sample paths $\x$ and $\x'$ are $B$--equivalent if $\Lambda\circ\bnd(\x)=\Lambda\circ\bnd(\x')$. 
If $\zeta$ is a countable partition of the space of sample paths, for any $\xi$ in $B$ we set  
$$
H_{\pp}(\zeta|\xi):=-\sum_i\pp^{\xi}(\zeta_i)\log\pp^{\xi}(\zeta_i).
$$
We define the \emph{conditional entropy} as
$$
H_{\pp}(\zeta| B)=\int_{B}H(\zeta|\xi) \ d\la(\xi).
$$

We need the following monotonicity property for the relative entropy, which goes back to \cite{Rokhlin52}. 

\begin{lemma}\label{lem:entropy properties}
Let $(B,\lambda)$ be a $\mu$--boundary. If $\zeta$ is a countable partition, then
$$
H_{\pp}(\zeta|B)\leq H_{\pp}(\zeta).
$$
\end{lemma}

\subsection{Relative entropy}
Let $\p \s$ be the space of all infinite sequences of elements of $W$: 
$$
\p \s =\Big\{\{g_n\}_{n\geq 1}\ :\ g_n\in W \mbox{ for } n\in\N\Big\}.
$$
Each element of $\p \s$ corresponds to a geodesic in the Cayley graph of $\Sigma$ with the standard generating set. 
As usual, one defines a topology on $\Sigma \cup \p \s$ by saying that a sequence $\{w_n\}_{n \geq 0} \subseteq \Sigma$  
of finite words converges to an infinite word $\g \in \p \s$ if any finite prefix of $\g$ is also a prefix of $w_n$ for all $n$ sufficiently large.
Thus, we will think of $\p \s$ as a geometric boundary of $\Sigma$. 

Let us now pick a measure $\mu$ on $\Sigma$, and consider the \emph{random walk} defined by picking a random sequence $\{g_n\}_{n \geq 1}$ 
of elements of $\Sigma$ independently with distribution $\mu$, and consider the product
$$x_n:= g_1 \dots g_n$$
Since there is no backtracking in the free semigroup, almost every sample path $\{x_n\}_{n\geq 0})$ converges to a point in $\p \s$. 
This defines a boundary map 
$$\Lambda : \Omega \to \p \s$$
$$\Lambda(\{ g_n \}_{n \geq 1}) := \lim_{n \to \infty} g_1 g_2 \dots g_n$$
and the pushforward measure $\lambda := \Lambda_\star \pp$ is called the \emph{hitting measure} for the random walk. 
Thus, the space $(\p \s, \lambda)$ is a $\mu$-boundary for the random walk $(\Sigma, \mu)$. 
Finally, for each $\g \in \p \s$ we denote the conditional measure of $\pp$ with respect to $\g$ as $\pp^{\g}$.

In order to study the Poisson boundary of the random walk $(\s,\mu)$, we will recast the conditional random walk with respect to the $\mu$--boundary $(\p \s,\la)$ in the framework of random walks on equivalence relations, which was introduced in \cite{Kaimanovich-Sobieczky2012}. 
Let us consider the equivalence relation on $\p \s$ given by the orbits for the action of $\s$; namely, two infinite words $\g=\{g_n\}_{n\geq 1}$ and $\g'=\{g'_n\}_{n\geq 1}$ are equivalent if their tails eventually coincide, that is there exist natural numbers $i$ and $j$ such that $g_{n+i}=g'_{n+j}$ for all $n\geq 0$. 

We define a Markov process on $\p \s$, called the \emph{leafwise Markov chain}, by setting for each $\g$ in $\p \s$ and each $w \in \s$ the transition probabilities
$$\pi(\g, w^{-1} \g):=p^{\g}(e, w) = \mu(w) \frac{d w \lambda}{d \lambda}(\g).$$
Note that sample paths for this process lie all the time in the same equivalence class, hence the process can be interpreted as a random walk along the equivalence relation. 
Note that $\pi$ defines for each $\g \in \partial \Sigma$ a Markov chain on $\s$, by setting for each $x, y \in \s$
$$p^{\g}(x, y) = \pi(x^{-1} \g, y^{-1} \g)$$
and by construction this precisely equals the conditional random walk defined in eq. \eqref{E:cond}.
Equivalently, this process can also be seen as a special case of a random walk with random environment $\p \s$ (but we will not use this language): an infinite word $\g$ is picked 
randomly according to the law of $\la$, and this determines the Markov process $p^{\g}$.

For each $n$, one defines the entropy of the $k^{th}$-step distribution as
$$
H_k(\g) =-\int_{\Omega}\log{\pp^{\g}(y_k = x_k)} \ d\pp^{\g}(\x)
$$ 
where $\pp^{\g}(y_k = x_k)$ means $ \pp^{\g}(\{ (y_n) \in \Omega \ : \ y_k = x_k\})$. Moreover, we set
$$H_k = \int_{\p \s} H_k(\g) \ d\lambda(\g)$$
the average entropy of the $k^{th}$ step.
In the language of partitions, we have 
$$
H_k(\g)=H_{\pp}(\alpha_k|\g)\ \ \mbox{and }\ \  H_k=H_{\pp}(\alpha_k|\p \s).
$$

We will apply the following entropy criterion for random walks along equivalence classes, due to Kaimanovich-Sobieczky.

\begin{theorem}[\cite{Kaimanovich-Sobieczky2012}]           \label{thm:KS}
If $H_1 < \infty$, then all entropies $H_n$ are finite, and there exists the limit 
$$h = \lim_{k\to \infty} \frac{H_k}{k} < \infty.$$
Moreover, $h = 0$ if and only if for $\lambda$-a.e. point $\g \in \p \s$ the Poisson boundary of the leafwise Markov chain is trivial. 
\end{theorem}

Let us call $h$ the \emph{relative asymptotic entropy}. By combining the theorem with the previous observations, we get the following.

\begin{theorem}\label{thm:trivial conditional bnd}
If $H_1 < \infty$ and the relative asymptotic entropy $h$ is zero, 
then $(\p \s, \lambda)$ is a model for the Poisson boundary of $(\s, \mu)$.
\end{theorem}

\begin{proof}
When $h=0$, Theorem~\ref{thm:KS} implies that  for $\la$--almost every infinite word $\g$ in the $\mu$--boundary $(\p \s,\la)$, the Poisson boundary associated with the conditional random walk $\pp^{\g}$ is trivial. Therefore, by Lemma~\ref{lem:maximality}, the $\mu$--boundary $(\p \s,\la)$ is indeed the Poisson boundary.

\end{proof}

\section{First criterion: finite entropy of the projection to $\mathbb{N}$} \label{S:main1}

In this section, we will prove our first criterion to identify the Poisson boundary (Theorem \ref{thm:main homomorphism}), which readily implies 
Theorem \ref{thm:main either or} from the introduction.

Consider the semigroup homomorphism 
$$
\begin{array}{c}
\phi:\s\to\mathbb{N}\\
w\to |w|.
\end{array}
$$
Let $\mu_{\phi}$ be the image of the probability measure $\mu$ under $\phi$. Since $\phi$ is a semigroup homomorphism, the $n$--fold convolution  of $\mu$ is mapped to the $n$--fold convolution  of $\mu_\phi$, which means 
\begin{equation}\label{eq:homoemorphis}
(\mu^{\star n})_{\phi}=(\mu_{\phi})^{\star n} \qquad \textup{for any }n.
\end{equation}
This implies that for any sample path $\x=\{x_n\}_{n\geq0}$ with respect to the random walk $(\s,\mu)$, its image  $\phi(\x):=\{\phi(x_n)\}_{n\geq0}$ is a sample path with respect to the random walk $(\mathbb{N},\mu_\phi)$.

We now define for each $k$ a partition on the space of sample paths, by setting two sample paths $\x$ and $\x'$ to be \emph{$\phi_k$--equivalent} if $|x_k|=|x'_k|$, or equivalently $\phi(x_k)=\phi(x'_k)$. 
\begin{lemma}\label{lem:zero phi}
If  $H_{\pp}(\phi_1)$ is finite, then $\lim_k\frac{1}{k}H(\phi_k)=0$.
\end{lemma}
\begin{proof}
Let $\phi\circ\pp$ be the probability measure measure on the sample paths associated to $\mu_\phi$. 
Equation \eqref{eq:homoemorphis} and the definition of entropy for partitions implies that
\begin{equation}\label{eq:entropy projection}
H_{\pp}(\phi_k)=H_{\phi\circ\pp}(\alpha_k).
\end{equation}
By Theorem~\ref{thm:abelian zero entropy}, we have the desired result.
\end{proof}

\begin{proposition}
If $\mu_\phi$ has finite entropy, then the relative entropy $h$ is zero.
\end{proposition}

\begin{proof}
We say two sample paths $\x$ and $\x'$ are $\p \s$--equivalent if and only if $\x$ and $\x'$ lie on the same infinite word. Note that if two sample paths $\x$ and $\x'$ are $\p \s$--equivalent and at the same time $\phi_k$--equivalent for some $k$, then $\x=\x'$. This is due to the lack of cancellations in the random walk $(\s,\mu)$, therefore if $\g=\{g_n\}_{n\geq1}$, we have
$$
\pp^{\g}\Big\{\x :\ |x_k|=n\Big\}=\pp^{\g}\Big\{\x\ :\ x_k=g_1\cdots g_n\Big\},
$$
which implies 
\begin{equation}
H_{\pp}(\phi_k|\g)=H_k({\g}).
\end{equation}
Taking the integral with respect to $\la$ on both sides yields
\begin{equation}\label{eq:coincidece of entropy}
H_{\pp}(\phi_k|\p \s)=\int_{\p \s}H(\phi_k|\g) \ d\la(\g)=\int_{\p \s}H_k({\g})\ d\la(\g)=H_k.
\end{equation}
Therefore, combining it with Lemma~\ref{lem:entropy properties} implies that
$$
H_k=H_{\pp}(\phi_k|\p \s)\leq H_{\pp}(\phi_k).
$$
Since $H_{\pp}(\phi_1)$ is finite, applying Lemma~\ref{lem:zero phi} yields
\begin{equation}\label{eq:abelian0}
h = \lim_k\frac{H_k}{k}\leq\lim_k\frac{H_{\pp}(\phi_k)}{k}=0
\end{equation}
as claimed.
\end{proof}

By virtue of Theorem~\ref{thm:trivial conditional bnd} and the previous proposition we obtain the main result of this section:

\begin{theorem}\label{thm:main homomorphism}
If the measure $\mu_\phi$ on $\mathbb{N}$ has finite entropy, then $(\p \s,\la)$ is the Poisson boundary of the random walk $(\s,\mu)$.
\end{theorem}

We now see a few corollaries. In particular, it is sufficient to assume that the original measure $\mu$ on $\Sigma$ has finite entropy. 

\begin{corollary}\label{thm:main entropy}
If $\mu$ has finite entropy, then $(\p \s,\la)$ is the Poisson boundary of the random walk $(\s,\mu)$.
\end{corollary}

\begin{proof}
Since the partition $\alpha_1$ is a subpartition  of the partition $\phi_1$, we have  
$$H_{\pp}(\phi_1)\leq H_{\pp}(\alpha_1)<\infty.$$ Therefore,
$\mu_\phi$ has finite entropy and Theorem~\ref{thm:main homomorphism} holds.
\end{proof}
\subsection{Logarithmic moment}
Let us define the \emph{logarithmic moment} of $\mu$ as
$$
L(\mu):=\sum_g\mu(g)\log|g|.
$$

The following elementary calculation shows that on $\mathbb{N}$ finite logarithmic moment implies finite entropy. 

\begin{lemma}[\cite{Der80}]  \label{lem : finite entropy}
Let $\theta$ be a probability measure on $\Bbb N$.  
If $L(\theta)=\sum_n\theta(n)\log n$ is finite, then $\theta$ has finite entropy, and 
$$
H(\theta)\leq2L(\theta)+c,
$$
where $c=2\sum_n\frac{\log n}{n^2}+1$.
\end{lemma}
\begin{proof}
Let 
$$
A=\left\{n\ :\ \frac{1}{\theta(n)}\leq n^2\right\}. 
$$
We can write
$$
H(\theta)=-\sum_{n\in A}\theta(n)\log{\theta(n)}-\sum_{n\in A^c}\theta(n)\log{\theta(n)}.
$$
The first term is bounded by $2L(\theta)$. We will show that the second term is bounded too. We know the function $-t\log t$ is increasing for $t\leq e^{-1}$ and 

$$
e^{-1}=-e^{-1}\log e^{-1}=\max\{-t\log t \ :\ t\in[0,1]\}.
$$
If $n>1$ and $n\in A^c$, then $\theta(n)<\frac{1}{n^2}< e^{-1}$. We have $-\theta(n)\log \theta(n)\leq -\frac{1}{n^2}\log{\frac{1}{n^2}}$. Therefore,
$$
-\sum_{n\in A^c}\theta(n)\log{\theta(n)}\leq \theta(1)\log \theta(1)+ 2\sum_{n=1}^{\infty}\frac{\log n}{n^2} \leq e^{-1}+2\sum_{n=1}^{\infty}\frac{\log n}{n^2}.
$$
\end{proof}

As another corollary of Theorem~\ref{thm:main homomorphism}, we recover the following result of Kaimanovich and the first author \cite{BK2013}.

\begin{corollary}\label{thm:main log}
If $\mu$ has finite logarithmic moment, then $(\p \s,\la)$ is the Poisson boundary of the random walk $(\s,\mu)$.
\end{corollary}

\begin{proof}
Since $\mu$ has finite logarithmic moment, so does its image under $\phi$, since
$$L(\mu)=\sum_g\mu(g)\log|g|=\sum_{n \in \mathbb{N}} \mu_\phi(n)\log n=L(\mu_\phi).$$
By Lemma~\ref{lem : finite entropy}, we know $H(\mu_\phi)$ is finite. So the condition in Theorem~\ref{thm:main homomorphism} holds. 
\end{proof}

Combining Corollary \ref{thm:main entropy} and \ref{thm:main log} completes the proof of Theorem \ref{thm:main either or} in the introduction.

\subsection{A remark on measurable partitions}
Another way to understand the previous argument is in term of measurable partitions; this will also clarify where we need some finite entropy assumption, 
as it is tempting to conclude that we do not. 
Let for any $n$ define the partition $\eta_n$ on $\Omega$ by saying that $\x \overset{\eta_n}{\sim} \x'$ if $x_k = x'_k$ for any $k \geq n$.
Then the tail partition for $(\Sigma, \mu)$ is $\eta = \bigwedge_{n = 1}^\infty \eta_n$, and the claim that the Poisson boundary is the space of infinite words is equivalent 
to 
$$\eta = \xi$$
(mod $0$), where $\xi$ is the partition given by two sample paths being equal when their limits in $\partial \s$ are the same.
One can rephrase the earlier proof by defining the partition $\eta^N_n$ by taking two paths as being in the same class if they have the same tail when projected to $\mathbb{N}$: namely, $\x \overset{\eta^N_n}{\sim} \x'$ if $|x_k| = |x'_k|$ for any $k \geq n$. Now, it is easy to see that $\eta_n = \eta_n^N \vee \xi$ for any $n$; moreover, since the Poisson boundary of $(\mathbb{N}, \mu_{\phi})$ is trivial for \emph{any} measure, then 
$$\bigwedge_{n = 1}^\infty \eta_n^N = \epsilon$$
where $\epsilon$ is the trivial partition where all elements have measure either $0$ or $1$. Now, the claim we want to prove is that 
$$\bigwedge_{n = 1}^\infty \eta_n = \bigwedge_{n = 1}^\infty (\eta^N_n \vee \xi) \overset{?}{=} \left(\bigwedge_{n = 1}^\infty \eta^N_n \right) \vee \xi  = \epsilon \vee \xi = \xi$$
It turns out that in general, without any notion of finite entropy, the identity 
$$\bigwedge_{n = 1}^\infty (\alpha_n \vee \beta) \overset{?}{=} \left(\bigwedge_{n = 1}^\infty \alpha_n \right) \vee \beta$$
is not true (not even mod $0$) for an arbitrary measurable partitions, not even when $\{\alpha_n\}_{n \geq 1}$ is a decreasing sequence such that $\bigwedge \alpha_n$ is the trivial partition.

In fact, following \cite{Hanson}, let us consider $\Omega = \{0, 1\}^\mathbb{N}$ the space of sequences $\x = \{ x_k \}_{k \geq 0}$ with product measure $(\frac{1}{2}, \frac{1}{2})^\mathbb{N}$. 
Let $\alpha_n$ be the partition defined by $\x \overset{\alpha_n}{\sim} \x'$ if $x_k = x'_k$ for every $k \geq n$, and $\beta$ the partition defined by $\x \overset{\beta}{\sim} \x'$ if either $\x = \x'$ or $\x = 1 - \x'$. 
Then for each $n$ 
$$\alpha_n \vee \beta = \aleph$$
the point partition $\aleph$ where each class is a singleton, while $\alpha_{n+1} \leq \alpha_n$ and $\bigwedge_{n = 1}^\infty \alpha_n = \epsilon$ the trivial partition. Thus, 
$$\epsilon = \bigwedge_{n = 1}^\infty (\alpha_n \vee \beta) \neq \left(\bigwedge_{n = 1}^\infty \alpha_n \right) \vee \beta = \beta$$
hence the identity does not hold.

\section{Stopping times and induced random walks} \label{S:stop}

Let us fix a finite word $w \neq e$ in the support of the probability measure $\mu$, and let $\delta_w$ be the probability measure concentrated at $w$.
For each sample path $\x$ with increments $\{g_n\}_{n \geq1}$, define $\tau_w$ as the first time the finite word $w$ appears as an increment; that is,
$$
\tau_w(\x):=\min\{i>0 : g_i=w\}.
$$
Note that since $\mu(w)>0$,  the finite word appears infinitely many times for $\mu^{\Bbb N}$--almost every sequence of increments. 
Therefore, $\tau_w$ is an almost surely finite stopping time. Let us define the \emph{first return measure} $\mu_w$ on $\Sigma$ as 
$$
\mu_{w}(g):=\pp\{\x : x_{\tau(\x)}=g\}.
$$ 
The usefulness of this construction comes from the following observation.
\begin{lemma}
The hitting measures of the random walks $(\Sigma, \mu)$ and $(\Sigma, \mu_w)$ on $\partial \Sigma$ are the same. 
\end{lemma}

\begin{proof}
Let us fix $w \in \Sigma$ such that $\mu(w) > 0$, and consider the stopping time $\tau = \tau_w$ defined above. Then for each $n$ one defines $\tau_0 = 0$ and 
recursively for $n \geq 1$
$$\tau_{n+1} := \tau_{n} + \tau(U^{\tau_n} \x)$$
Then, we have the almost everywhere defined map 
$$i_w : (\Sigma^\mathbb{N}, \mu^\mathbb{N}) \to (\Sigma^\mathbb{N}, \mu_w^\mathbb{N})$$
$$i_w(\{ g_n \}_{n \geq 1}) := \{ x_{\tau_{n-1}}^{-1} x_{\tau_n}  \}_{n \geq 1}$$
which makes the following diagram of measurable maps commute: 
\[\begin{tikzcd}
(\Sigma^\mathbb{N},\mu^{\mathbb{N}}) \arrow[swap]{d}{i_w}  \arrow{r}{\Phi}  & \p \s   \\
(\Sigma^\mathbb{N},\mu_w^{\mathbb{N}})   \arrow{ur} {\Phi} & \phantom{a}
\end{tikzcd}\]
where $\Phi(\{ g_n \}) := \lim_n g_1 g_2 \dots g_n$. Moreover, by construction $(i_w)_\star(\mu^\mathbb{N}) = \mu_w^\mathbb{N}$.
Hence, if $\lambda_w$ is the hitting measure for the random walk $(\Sigma, \mu_w)$ and $\lambda$ is the hitting measure for $(\Sigma, \mu)$, then 
$$\lambda_w = \Phi_\star (\mu_w^{\mathbb{N}}) = \Phi_\star (i_w)_\star (\mu^\mathbb{N}) = \Phi_\star (\mu^\mathbb{N}) = \lambda.$$
\end{proof}

For each function $f : \Omega \to \mathbb{R}$, let us denote as $E(f):=\int_{\Omega} f \ d\pp$ the expectation of $f$.
The following simple computation shows that the expectation of $\tau_w$ is finite.

\begin{lemma}\label{lem: finite w}
Let $\tau_w=\min\{n \geq 1: g_n=w\}$. If $\mu(w)>0$, then $E(\tau_w)$ is finite and is equal to $\frac1{\mu(w)}$.
\end{lemma}
\begin{proof}
We can write $\pp (\tau=n+1) =(1-\mu(w))^n\mu(w)$, hence
$$
E(\tau_w)=\sum_{n=0}^\infty (n+1)\pp(\tau=n+1)=\mu(w)\sum_{n=0}^\infty (n+1)\Big(1-\mu(w)\Big)^n=\frac{1}{\mu(w)}.
$$
\end{proof}

One of the key facts we will use is that the Poisson boundary for the new measure induced by the stopping time 
is equal to the Poisson boundary for the original measure:

\begin{proposition}[\cite{BK2013}] \label{pro:coincidence}
The Poisson boundary of $(\s, \mu)$ coincides with the Poisson boundary of $(\s, \mu_w)$.
\end{proposition}
\begin{proof}
We will show the equivalent claim that the spaces of bounded harmonic functions for $\mu$ and $\mu_w$ coincide. Observe that $\mu_w=\sum_{k \geq 0}\alpha^{\star k}\star\beta$ where $\beta=\mu(w)\delta_w$ and $\alpha=\mu-\beta$. Let $f$ be $\mu$-harmonic. Then by definition 
$$f = \mu \cdot f = \alpha \cdot f + \beta \cdot f$$
By acting with $\alpha^{\star k}$ on both sides, one gets 
$$\alpha^{\star k} \cdot f = \alpha^{\star(k+1)} \cdot f + (\alpha^{\star k} \star \beta) \cdot f$$
hence by summing over $k$ and using the telescoping series (since $\Vert \alpha^{\star k} \Vert \to 0$)
$$\sum_{k = 0}^\infty (\alpha^{\star k} \star \beta) \cdot f = \sum_{k = 0}^\infty (\alpha^{\star k} \cdot f - \alpha^{\star(k+1)} \cdot f) = f$$
hence $f$ is $\mu_w$--harmonic. 

Let now $f:\Sigma_w\to\Bbb R$ be a bounded $\mu_w$--harmonic function. We will extend $f$ to a bounded $\mu$--harmonic function; 
this extension is  similar to Furstenberg's proof for the invariance of the Poisson boundary for an induced random walk to a recurrent subgroup \cite{Fu70}. 
For any $g$ in $\Sigma$, let us define 
$$
F(g):=\sum_{y\in\Sigma_w}f(y)\theta_g(y)
$$ where 
$$
\theta_g(y)=\pp_g\{\x: x_{{\tau_w}(\x)}=y\}.
$$
Note that if $g$ is in $\Sigma_w$, then $F(g)=\sum_hf(gh)\mu_w(h)$ which is equal to $f(g)$ when $f$ is $\mu_w$--harmonic, therefore $F(g)=f(g)$.  
We claim that $F$ is $\mu$--harmonic.  First, observe that $\tau_w(g,gh_1,\cdots,gh_1\cdots h_n,\cdots)=n>1$ means that word $w$ appears as an increment  for the first time in the $n^{th}$ step, therefore $\tau_w(gh_1,\cdots,gh_1\cdots h_n,\cdots)=n-1$. So, we can write
$$
\theta_g(y)=\sum_h\sum_{n\geq1}\pp_{g}\{\x: x_1=gh, {\tau_w}(\x)=n,\  x_n=y\}=\mu(w)\delta_{gw}(y)+\sum_{h \neq w} \mu(h)\theta_{gh}(y).
$$
Multiplying both sides by $f(y)$ and summing over $y$ yields 
$$
F(g)=\sum_hF(gh)\mu(h)
$$
as needed.
\end{proof}

We also need the following Abramov-type formula, which generalizes Lemma 2.5 of \cite{Behrang2016}. 

\begin{proposition} \label{theo:finite moment}
Let $\mu$ be a probability measure on a semigroup $\Sigma$, and let $F \in L^1(\Sigma, \mu)$ be a non-negative function such that
$$
F(gh)\leq F(g)+F(h) \qquad \textup{for all }g, h \in \Sigma.
$$ 
Let $(\Omega, \pp)$ be the space of sample paths for the random walk $(\Sigma, \mu)$, and let $\tau: \Omega \to \mathbb{N}$ be a stopping time
in $L^1(\Omega, \pp)$. Then 
$$
\sum_w\mu_{\tau}(w)F(w)\leq E(\tau)E(F).
$$
\end{proposition}

\begin{proof}
Define $M_n(\x)=nE(F)-F(x_n)$. Let $\mathcal{A}_0^n$ be the $\sigma$-algebra generated by the first $n+1$ positions $x_0, x_1, \dots, x_n$ of the random walk $(\Sigma,\mu)$. We have
$$
E(M_{n+1}|\mathcal{A}_0^n)(\x)=(n+1)E(F)-\sum_hF(x_nh)\mu(h).
$$
Since $F(x_nh)\leq F(x_n) + F(h)$, 
$$
E(M_{n+1}|\mathcal{A}_0^n)(\x)\geq nE(F)-F(x_n)=M_n(\x),
$$
which means that the sequence $\{(M_n,\mathcal{A}_0^n)\}_{n\geq1}$ is a submartingale. Applying  Doob's optional stopping theorem to the stopping time $\tau\wedge n=\min\{\tau,n\}$ implies 
$$
0=E(M_1)\leq E(M_{\tau\wedge n})
$$
hence 
$$ \int F(x_{\tau\wedge n}) \ d\pp \leq E(\tau\wedge n)E(F) \leq E(\tau) E(F).$$
Note that because $\tau$ is almost surely finite, $\lim_n \pp(\tau>n)=0$. Hence, for  any finite word $h$, we have $\mu_{\tau\wedge n}(h)\to \mu_\tau(h)$ as $n$ goes to infinity, therefore, Fatou's lemma implies 
$$
\sum_h \mu_\tau (h) F(h) = \sum_h\lim_n\mu_{\tau\wedge n}(h)F(h)\leq\liminf_n \int F(x_{\tau\wedge n})\ d\pp\leq E(\tau)E(F).$$
\end{proof}

\begin{corollary}
Let $F$ and $\tau$ satisfy the same conditions as in the previous theorem. Then
$$
\lim_n\frac{F(x_{\tau_n})}{n}=E(\tau)E(F)
$$
for $\pp$--almost every sample path $\{x_n\}_{n\geq 0}$.
\end{corollary}

\section{Random walks with finite logarithmic $w$-moment} \label{S:main2}

We now get to the proof of the second main result, namely Theorem \ref{thm:finite w norm} from the introduction.

\subsection{The $w$--norm} \label{S:w}
Fix a finite word $w$ in a free semigroup $\Sigma$. For each finite word $g$ in $\Sigma$, define the \emph{w-norm} $|g|_w$ as the number of times the word $w$ appears 
as a subword of $g$, plus the length of $w$; more precisely, 
$$
|g|_w := \textup{card}\{g'\in\Sigma\ :\ g=g'wg'' \textup{ for some }g'' \in \Sigma\}+|w|.
$$
\begin{lemma}
The $w$--norm is subadditive, i.e. any two words $g_1, g_2 \in \s$
satisfy the inequality
$$|g_1 g_2|_w \leq |g_1|_w + |g_2|_w.$$
\end{lemma}
\begin{proof}
Let us define $I_w(g) :=  \textup{card}\{g'\in\Sigma\ :\ g=g'wg'' \textup{ for some }g'' \in \Sigma\}$.
If $w$ is a subword of $g_1 g_2$, then one has a decomposition $g_1g_2 = g'wg''$. Now, if $|g'| \leq |g_1| - |w|$, then 
$w$ is also a subword of $g_1$. Similarly, if  $|g'| \geq |g_1| +1$, then $w$ appears as a subword of $g_2$. 
Otherwise, there are at most $|w|$ possible choices for $|g'|$, which implies 
$$I_w(g_1 g_2) \leq I_w(g_1) + I_w(g_2) + |w|.$$
Adding $|w|$ to both sides yields the claim.
\end{proof}

Let $L_w(\mu)$ be the logarithmic moment with respect to the $w$--norm, that is 
$$L_w(\mu) :=\sum_g\mu(g)\log|g|_w.$$
Observe that for any finite word $g$ in $\s$, we have $|w|\leq |g|_w\leq |w|+|g|$, therefore,
if $L(\mu)$ is finite, then $L_w(\mu)$ also is finite.

Recall that a sequence $\{z_n\}_{n\geq 1}$ of random variables defined on the same measure space is \emph{stationary} if for each $k, n$ the law of the 
$n$-tuples $(z_1, z_2, \dots, z_n)$ and $(z_{k+1}, \dots, z_{k+n})$ is the same. 

\begin{lemma}[\cite{Behrang}, Lemma 3.6.4] \label{lem:stationary}
Let $\{z_n\}_{n\geq1}$ be a non-negative stationary process. If $\log(1+ z_1)$ has finite expectation, then
$$
\lim_k\frac{1}{k}\log(1+z_1+\cdots+z_k)=0
$$
almost surely and in $L^1$.
\end{lemma}

\begin{lemma}\label{lem:zero w-norm log}
If $L_w(\mu)$ is finite, then $\lim_k\frac{1}{k} L_w(\mu^{\star k})=0$.
\end{lemma}

\begin{proof}
Let $\{g_n\}_{n\geq1}$ be the increments of $\{x_n\}_{n\geq0}$, so for every $k$ one can write $x_k = g_1 \dots g_k$, hence by subadditivity
$$
\log |x_k|_w \leq\log(1 + |g_1|_w+\cdots+|g_k|_w).
$$
Applying Lemma~\ref{lem:stationary} implies the desired result.
\end{proof}

\begin{theorem}
Let $w \neq e$ be a finite word in $\s$ such that $L_w(\mu)$ is finite. Then, $(\p \s, \la)$ is the Poisson boundary of the random walk $(\s,\mu)$.
\end{theorem}

\begin{proof}
Since we can replace $\mu$ by $\mu^{\star n}$ without changing the Poisson boundary, we may without loss of generality assume that $\mu(w)>0$.  
Moreover, by Proposition~\ref{pro:coincidence}, it is enough to describe the Poisson boundary associated with the random walk $\mu_w$ induced by the stopping time $\tau_w$.
By Lemma \ref{lem: finite w} the expected stopping time $E(\tau_w)$ is finite, therefore, applying Proposition \ref{theo:finite moment} when $F(g):=\log(1+|g|_w)$ implies that the logarithmic $w$-moment is also finite and  
 $$L_w(\mu_{w}) \leq E(\tau_w) (L_w(\mu) + \log 2).$$
Let $\s_w$ be the free semigroup generated by the support of $\mu_{w}$, and let us 
denote by $\q$ the probability measure on the space $\Omega$ of sample paths with respect to $\mu_w$. 
Let us disintegrate $\q$ with respect to the system of conditional measures $\{\q^{\g}\}_{\g\in\partial \Sigma}$, so that
\begin{equation}\label{eq: disintegration}
\q=\int_{\partial \Sigma}\q^{\g} \ d\lambda(\g).
\end{equation}
Since  the random walk $(\Sigma_w,\mu_w)$ has finite logarithmic moment with respect to the $w$--norm, 
$$
L_w(\mu_w)=\sum_{x \in \Sigma} \mu_w(x)\log|x|_w=\int_{\Omega}\log|x_1|_w \ d\q(\boldsymbol{x})<\infty,
$$
which implies that $\lambda$--almost every conditional probability measure $\q^{\g}$ has finite logarithmic moment with respect to the $w$--norm, that is 
$$
L_1(\q^{\g}):=\int_{\Omega} \log|x_1|_w \ d\q^{\g}(\x)<\infty.
$$
Similarly, let us define for any $k$ and and $\g \in \p \s$
$$L_k(\q^{\g}) := \int_{\Omega} \log|x_k|_w \ d\q^{\g}(\x)$$
which by applying \eqref{eq: disintegration} satisfies
\begin{equation} \label{E:lk}
\int_{\p \s} L_k(\q^{\g})\ d\lambda(\g) = \int_{\Omega} \int_{\partial \Sigma}\log|x_k|_w \ d\q^{\g}(\x) d\lambda(\g) = L_w(\mu_w^{\star k}).
\end{equation}
Let us denote by  $H_k({\q}^{\g}) :=H_{\q}(\alpha_k|\g)$  the entropy of the $k^{th}$ step with respect to the conditional probability measure $\q^{\g}$,
namely
$$
H_k({\q}^{\g}) = - \sum_{x \in \Sigma} \q^{\g}(x_k = x) \log{\q^{\g}(x_k = x)} .
$$
Note that if two sample paths $\x$ and $\x'$ for the random walk $(\Sigma, \mu_w)$ lie on the same infinite word $\g$ 
and satisfy $|x_k|_w=|x'_k|_w=n$ for some $k$, then actually $x_k=x'_k$; therefore,  
$$
H_k({\q}^{\g}) =-\sum_n\q^{\g}\{\x :\ |x_k|_w=n\}\log\q^{g}\{\boldsymbol{x} :\ |x_k|_w=n\}
$$
Hence, by virtue of Lemma~\ref{lem : finite entropy}, we have
$$
H_k(\q^{\g})\leq 2L_k(\q^{\g})+c,
$$
hence, combining it with eq. \eqref{E:lk},
$$
H_k =\int_{\partial \Sigma}H_k(\q^{\g})\ d\lambda(\g)\leq 2\int_{\partial \Sigma}L_k(\q^{\g})\ d\lambda(\g)+c=2L_w(\mu_w^{\star k})+c.
$$
By Lemma~\ref{lem:zero w-norm log}, the relative asymptotic entropy vanishes, since
$$h = \lim_k \frac{H_k}{k} \leq \lim_k \frac{2 L_w(\mu_w^{\star k})+c}{k} = 0.$$
Consequently,  Theorem~\ref{thm:trivial conditional bnd} implies that 
$(\partial \s, \lambda)$ is the Poisson boundary.
\end{proof}

\bibliographystyle{alpha}
\bibliography{biblography}

\end{document}